\documentclass[a4paper,12pt]{amsart}

\usepackage{a4wide}
\usepackage{tikz}
\usepackage{xcolor}

\newif\ifdetails
\detailstrue
\newcommand{\DETAIL}[1]%
{\ifdetails\par\fbox{\begin{minipage}{0.9\linewidth}\textit{Detail:}
      #1\end{minipage}}\par\fi}
\newcommand{\TODO}[1]%
{\ifdetails\par\fbox{\begin{minipage}{0.9\linewidth}\textbf{TODO:}
      #1\end{minipage}}\par\fi}

\usepackage{makecell}
\usepackage{relsize}

\newtheorem{lemma}{Lemma}
\newtheorem{proposition}[lemma]{Proposition}
\newtheorem{theorem}[lemma]{Theorem}
\newtheorem{corollary}[lemma]{Corollary}
\theoremstyle{remark}

\newtheorem{conjecture}{Conjecture}
\newtheorem{question}{Question}

\newtheorem{problem}{Problem}

\usepackage{caption}
\usepackage{subcaption}
\usepackage{float} 
\usepackage[pdftex,a4paper,
citecolor = blue, colorlinks=true,urlcolor=blue]{hyperref}
\newcommand{\old}[1]{{}}

\title{Inducibility of $d$-ary trees}

\author[\'E. Czabarka, A. A. V. Dossou-Olory, L. A. Sz\'ekely]{\'Eva Czabarka, Audace A. V. Dossou-Olory, L\'aszl\'o A. Sz\'ekely}
\address{\'Eva Czabarka and L\'aszl\'o A. Sz\'ekely\\ Department of Mathematics \\ University of South Carolina \\ Columbia, SC 29208 \\ USA}
\email{\{czabarka,szekely\}@math.sc.edu }
\thanks{The second author was supported by Stellenbosch University and African Institute for Mathematical Sciences (AIMS) South Africa, the third author was supported in part by the  NSF DMS,  grant number 1600811, the fourth author was supported by the National
Research Foundation of South Africa, grant number 96236.}
\author[S. Wagner]{\\Stephan Wagner}
\address{Audace A. V. Dossou-Olory and Stephan Wagner\\ Department of Mathematical Sciences \\ Stellenbosch University \\ Private Bag X1, Matieland 7602 \\ South Africa}
\email{\{audaced,swagner\}@sun.ac.za}

\begin{document}

\begin{abstract}
Imitating a recently introduced invariant of trees, we initiate the study of the inducibility of $d$-ary trees (rooted trees whose vertex outdegrees are bounded from above by $d\geq 2$) with a given number of leaves. We determine the exact inducibility for stars and binary caterpillars. For $T$ in the family of strictly $d$-ary trees (every vertex has $0$ or $d$ children), we prove that the difference between the maximum density of a $d$-ary tree $D$ in $T$ and the inducibility of $D$ is of order $\mathcal{O}(|T|^{-1/2})$ compared to the general case where it is shown that the difference is $\mathcal{O}(|T|^{-1})$ which, in particular, responds positively to an existing conjecture on the inducibility in binary trees. We also discover that the inducibility of a binary tree in $d$-ary trees is independent of $d$. Furthermore, we establish a general lower bound on the inducibility and also provide a bound for some special trees. Moreover, we find that the maximum inducibility is attained for binary caterpillars for every $d$.
\end{abstract}

\maketitle

\section{Introduction}

The inducibility of graphs is an old topic in extremal graph theory. However, the inducibility of trees with a given number of leaves is a novel concept that has been put forward recently in \cite{czabarka2016inducibility} to deal with a typical problem stemming from the phylogenetics context in mathematical biology. In a compact way, the problem that is addressed in \cite{czabarka2016inducibility} can be stated as maximizing the asymptotic density of appearances of `small' binary trees in binary trees with `large' number of leaves.

The present work is aimed at extending the concept from binary trees to rooted trees with bounded degrees in which no vertex has degree $2$ (except possibly for the root).

\medskip
The problem of maximizing the density of graphs in larger graphs is a key concept that has been explored starting from the work \cite{pippenger1975inducibility} of Golumbic and Pippenger and has ever since been of interest to graph theorists. In \cite{pippenger1975inducibility}, the initiators introduced the subject for simple graphs: for finite and simple graphs $G$ and $H$ with $k$ and $n$ vertices, respectively, let $\mathcal{I}(G,H)$ be the number of distinct subgraphs induced by $k$ distinct vertices of $H$ which are isomorphic to $G$. The limit
\begin{align*}
I(G):=\lim_{n\to \infty} \Bigg( \max_{|H|=n}\mathcal{I}(G,H)/\binom{n}{k}\Bigg),
\end{align*}
where the maximum runs over all finite and simple graphs on $n\geq k$ vertices, always exists and is called the \textit{inducibility} of the graph $G$. In this setting, it is the asymptotic behavior of the maximum number of appearances of $G$ as subgraph in an arbitrary $n$-vertex graph as $n\to \infty$ which is captured by the inducibility $I(G)$.

\medskip
There have been numerous investigations on $I(G)$ in special cases. Indeed, the maximum number of induced subgraphs of $H$ isomorphic to the complete bipartite graph $K_{k,k}$ has been studied in a 1986 paper \cite{bollobas1986maximal} by Bollob\'{a}s et al. A great deal of work (see \cite{exoo1986dense, brown1994inducibility, bollobas1995maximal}) followed \cite{bollobas1986maximal} afterwards. As such, Brown and Sidorenko~\cite{brown1994inducibility} computed  explicitly $I(K_{k,k+l})$ for all $l\geq 1$ but under the restriction $k \geq l(l-1)/2$. For $k < l(l-1)/2$, they stated the result as a function of the maximum over $[0,1]$ of a certain polynomial in a single variable.

In 2014, James Hirst \cite{hirst2014inducibility} determined, employing Razborov's flag algebra method and semi-definite programming techniques, the inducibility of two 4-vertex graphs: the complete tripartite graph $K_{1,1,2}$ and the so-called paw graph (graph constructed from a triangle by appending a pendant edge). The concept of inducibility is still gaining consideration from several research groups; see \cite{hatami2014inducibility} and \cite{even2015note} for some recent results on blow-up of graphs and graphs on four vertices, respectively. The language of flag algebra was also employed recently in \cite{balogh2016maximum} to derive the inducibility of the cycle on five vertices, thereby settling a particular case of a conjecture formulated in \cite{pippenger1975inducibility}. 

\medskip
Let us mention that there has also been interest in oriented graphs. For instance, in 2011, Sperfeld \cite{sperfeld2011inducibility} explored, by means of flag algebra, the inducibility of (monodirected) graphs with at most four vertices. Three years later, Huang \cite{huang2014maximum} determined, by means of a different approach, the maximum induced proportion of directed star graphs.

\section{Inducibility of Trees}

One of the most widely used classes of graphs is the class of trees. Due to their prevalence in many situations, the study of the inducibility of trees emerged in 2016, and has been addressed in two different settings.

In \cite{bubeck2016local}, Bubeck and Linial investigated what they called \textit{$k$-profile of trees}, which is the vector of induced proportions of trees with $k$ vertices classified according to their isomorphism type. At the end of their paper, Bubeck and Linial briefly defined a notion of inducibility of trees in this specific situation. Analogously, originally motivated by a question from phylogenetics, the recent paper \cite{czabarka2016inducibility} introduced a further variant of the concept by studying the inducibility for rooted binary trees with a given number of leaves.

\medskip
A natural related problem for further study is the inducibility of degree-bounded rooted trees. The special case of binary trees is also known as phylogenetic trees---these trees are widely used in biology to describe how species are evolutionarily linked \cite{hafner1988phylogenetic,semple2003phylogenetics}. The present work aims at producing a number of results on the inducibility of degree-bounded rooted trees. It also provides an affirmative answer to a conjecture from \cite{czabarka2016inducibility}.

\medskip
Let $d\geq 2$ be an arbitrary but fixed positive integer. A rooted tree will be called a \textit{$d$-ary tree} in this paper if each of its non-leaf vertices has between $2$ and $d$ children. For a rooted tree $D$, we write $|D|$ for the number of leaves of $D$. 

If $S$ is a subset of the leaf set of a $d$-ary tree $T$, then the unique subtree obtained by
\begin{itemize}
\item first extracting the minimal subtree of $T$ containing the leaves in $S$, and then 
\item deleting, except possibly for the root, the vertices of degree $2$ in the induced tree 
\end{itemize}
is called a \textit{leaf-induced subtree} of $T$, see Figure~\ref{leaf-induced}. We define the root of the resulting tree as the most recent common ancestor of the leaves in $S$. By a \textit{copy} of $D$ in $T$, we mean any leaf-induced subtree of $T$ isomorphic (in the sense of rooted trees) to $D$. The total number of copies of $D$ in $T$ will be denoted by $c(D,T)$, and the quotient 
\begin{align*}
\frac{c(D,T)}{\binom{|T|}{|D|}}
\end{align*}
by $\gamma(D,T)$. In words, $\gamma(D,T)$ is the proportion of all subsets of $|D|$ leaves of $T$ that induce a copy of $D$.

\begin{figure}[!h]\centering  
\begin{tikzpicture}[thick]
\node [circle,draw] (r) at (0,0) {};

\draw (r) -- (-2,-2);
\draw (r) -- (0,-4);
\draw (r) -- (2,-2);
\draw (-2,-2) -- (-2.5,-4);
\draw (-2,-2) -- (-1.5,-4);
\draw (2,-2) -- (1,-4);
\draw (2,-2) -- (2,-4);
\draw (2,-2) -- (3,-4);
\draw (1.25,-3.5) -- (1.5,-4);

\node [fill,circle, inner sep = 2pt ] at (-2.5,-4) {};
\node [fill,circle, inner sep = 2pt ] at (-1.5,-4) {};
\node [fill,circle, inner sep = 2pt ] at (0,-4) {};
\node [fill,circle, inner sep = 2pt ] at (1,-4) {};
\node [fill,circle, inner sep = 2pt ] at (1.5,-4) {};
\node [fill,circle, inner sep = 2pt ] at (2,-4) {};
\node [fill,circle, inner sep = 2pt ] at (3,-4) {};

\node at (-1.5,-4.5) {$\ell_1$};
\node at (0,-4.5) {$\ell_2$};
\node at (1,-4.5) {$\ell_3$};
\node at (2,-4.5) {$\ell_4$};

\node [circle,draw] (r1) at (6,-2) {};

\draw (r1) -- (5,-4);
\draw (r1) -- (6,-4);
\draw (r1) -- (7,-4);
\draw (6.75,-3.5)--(6.5,-4);

\node [fill,circle, inner sep = 2pt ] at (5,-4) {};
\node [fill,circle, inner sep = 2pt ] at (6,-4) {};
\node [fill,circle, inner sep = 2pt ] at (6.5,-4) {};
\node [fill,circle, inner sep = 2pt ] at (7,-4) {};

\node at (5,-4.5) {$\ell_1$};
\node at (6,-4.5) {$\ell_2$};
\node at (6.5,-4.5) {$\ell_3$};
\node at (7,-4.5) {$\ell_4$};

\end{tikzpicture}
\caption{A ternary tree and the subtree induced by four leaves ($\ell_1,\ell_2,\ell_3,\ell_4$).}\label{leaf-induced}
\end{figure}
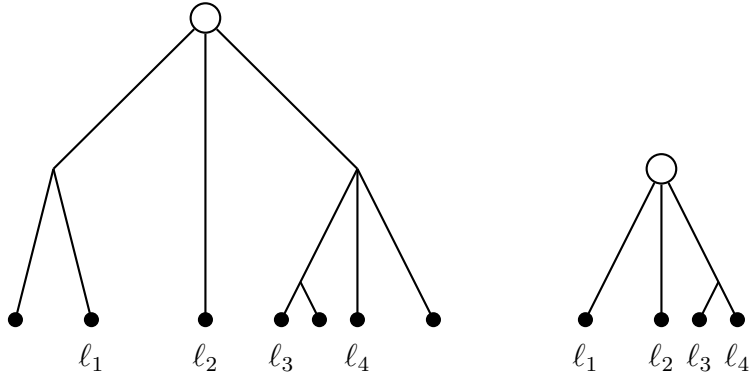

We define the \textit{inducibility of a $d$-ary tree $D$ in $d$-ary trees} $T$ as being the limit superior, taken over all $d$-ary trees, of the density of all subsets of $|D|$ leaves of $T$ that induce a copy of $D$:
\begin{align}\label{Formula of the i(B)}
I_d(D):= \limsup_{\substack{|T|\to \infty \\ T~\text{$d$-ary tree}}} \gamma(D,T) \, ,
\end{align}
where we mean 
$$
\limsup_{\substack{|T|\to \infty \\ T~\text{$d$-ary tree}}} \gamma(D,T)= \limsup_{n\rightarrow \infty} \max_{\substack{|T|=n \\ T~\text{$d$-ary tree}}} \gamma(D,T).
$$
This novel concept of inducibility is, of course, in analogy to the very first one \cite{pippenger1975inducibility} for simple graphs. However, it is not clear at this point whether the sequence
\begin{align*}
\Big(\max_{\substack{|T|=n \\ T~\text{$d$-ary tree}}}\gamma(D,T)\Big)_{n\geq 1}
\end{align*} 
converges for every $d$-ary tree $D$, in which case we could simply write
\begin{align*}
I_d(D)=\lim_{n\to \infty} \max_{\substack{|T|=n \\ T~\text{$d$-ary tree}}}\gamma(D,T)\,.
\end{align*} 

But notice that equation \eqref{Formula of the i(B)} already tells us that the maximum number of copies of $D$ in an arbitrary $n$-leaf $d$-ary tree is at most 
\begin{align*}
\binom{n}{|D|} \big(I_d(D) +o(1)\big)
\end{align*}
as $n\to \infty$. In fact, we shall prove later that $I_d(D)$ is always a positive real number for every $d$-ary tree $D$---see Proposition~\ref{howsamllJS}.

For the particular case where $d=2$, it was conjectured in \cite{czabarka2016inducibility} that the asymptotic formula
\begin{align} \label{errorterm1}
\max_{\substack{|T|=n\\ T~\text{binary tree}}}\gamma(B,T) =I_2(B) +\mathcal{O}(n^{-1})
\end{align} 
holds for every binary tree $B$. In the present work, we affirm this conjecture, and even generalize the result
 for every $d$ (Theorem~\ref{maxdensityId}).

\medskip
A $d$-ary tree in which every vertex has exactly $0$ or $d$ children will be called a \textit{strictly} $d$-ary tree. It turns out that $I_d(D)$ can also be computed by merely taking the limit superior in the family of strictly $d$-ary trees. 

Define the \textit{inducibility of a $d$-ary tree $D$ in strictly $d$-ary trees} as being the limit superior, taken over all strictly $d$-ary trees, of the density of subsets of $|D|$ leaves that induce a copy of $D$. From now we always assume that $d\geq 2$ and $|D|\geq 2$.
That is, we have
\begin{align*}
i_d(D):= \limsup_{\substack{|T|\to \infty \\T~\text{strictly $d$-ary tree}}} \gamma(D,T)\,.
\end{align*}
We note that $0\leq i_d(D)\leq I_d(D)\leq 1$ by definition.
We demonstrate in Section~\ref{IndVsMax} that the identity $I_d(D)=i_d(D)$ holds for every $d$-ary tree $D$ and every $d$---see Theorem~\ref{towards idId}. 
As counterpart of \eqref{errorterm1}, we show in Corollary~\ref{lowerBoundAsymptid} that
\begin{align*}
\max_{\substack{|T|=n\\T~\text{strictly $d$-ary tree}}} \gamma(D,T) = i_d(D)+ \mathcal{O}(n^{-1/2}).
\end{align*}
In Section~\ref{bin} we show that inducibility of any fixed binary tree is the same among binary trees and $d$-ary trees, for any $d\geq 2$,
in Section~\ref{generic} we show that the inducibility of any $d$-ary tree is positive, and finally in Section~\ref{maxim} we show that the only $d$-ary trees that have inducibility 1 are the binary caterpillars.
 
\section{On Stars and Binary Caterpillars}

For technical reasons, it is consistent and convenient to treat the single vertex as both leaf and root. By joining leaves to a single vertex, one obtains a \textit{star}. We denote the $k$-leaf star by $C_k$ as suggested by Fig.~\ref{Stars Ck.}. Obviously, our definitions bring
$I_d(C_2)=i_d(C_2)=1$
for every $d$. 

\begin{figure}[htbp]\centering
  \begin{subfigure}[b]{0.2\textwidth} \centering  
\begin{tikzpicture}[thick,level distance=7mm]
\tikzstyle{level 1}=[sibling distance=10mm]
\node [circle,draw]{}
child {[fill] circle (2pt)}
child {[fill] circle (2pt)};
\end{tikzpicture}
  \caption{$C_2$}
  \end{subfigure}\quad
  \begin{subfigure}[b]{0.2\textwidth} \centering 
\begin{tikzpicture}[thick,level distance=10mm]
\tikzstyle{level 1}=[sibling distance=8mm]
\node [circle,draw]{}
child {[fill] circle (2pt)}
child {[fill] circle (2pt)}
child {[fill] circle (2pt)};
\end{tikzpicture}
  \caption{$C_3$}
  \end{subfigure}\qquad
   \begin{subfigure}[b]{0.2\textwidth} \centering   
\begin{tikzpicture}[thick,level distance=14mm]
\tikzstyle{level 1}=[sibling distance=8mm]
\node [circle,draw]{}
child {[fill] circle (2pt)}
child {[fill] circle (2pt)}
child {[fill] circle (2pt)}
child {[fill] circle (2pt)};
\end{tikzpicture}
   \caption{$C_4$}
  \end{subfigure}\qquad
   \begin{subfigure}[b]{0.2\textwidth} \centering  
\begin{tikzpicture}[thick,level distance=16mm]
\tikzstyle{level 1}=[sibling distance=6mm]
\node [circle,draw]{}
child {[fill] circle (2pt)}
child {[fill] circle (2pt)}
child {[fill] circle (2pt)}
child {[fill] circle (2pt)}
child {[fill] circle (2pt)};
\end{tikzpicture}
  \caption{$C_5$}
  \end{subfigure}
\caption{Stars $C_k$.} \label{Stars Ck.}
\end{figure}
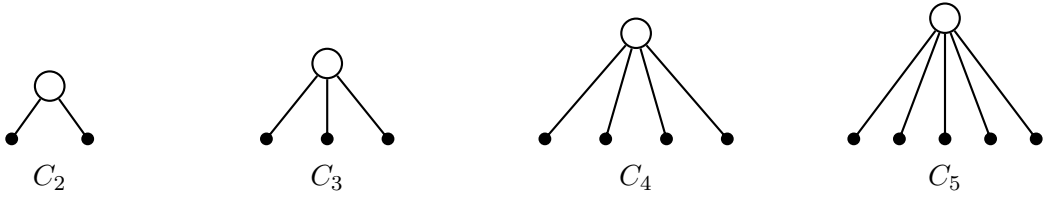

A \textit{complete} $d$-ary tree is a strictly $d$-ary tree in which all the leaves ($d^h$ in total) are at the same distance $h$ from the root. See Fig.~\ref{The complete 4 ary tree of height $2$.} for the complete $4$-ary tree of height $2$.

\begin{figure}[!h]\centering  
\begin{tikzpicture}[thick,level distance=16mm]
\tikzstyle{level 1}=[sibling distance=27mm]
\tikzstyle{level 2}=[sibling distance=6mm]
\node [circle,draw]{}
child {child {[fill] circle (2pt)}child {[fill] circle (2pt)}child {[fill] circle (2pt)}child {[fill] circle (2pt)}}
child {child {[fill] circle (2pt)}child {[fill] circle (2pt)}child {[fill] circle (2pt)}child {[fill] circle (2pt)}}
child {child {[fill] circle (2pt)}child {[fill] circle (2pt)}child {[fill] circle (2pt)}child {[fill] circle (2pt)}}
child {child {[fill] circle (2pt)}child {[fill] circle (2pt)}child {[fill] circle (2pt)}child {[fill] circle (2pt)}};
\end{tikzpicture}
\caption{The complete $4$-ary tree of height $2$.}\label{The complete 4 ary tree of height $2$.}
\end{figure}
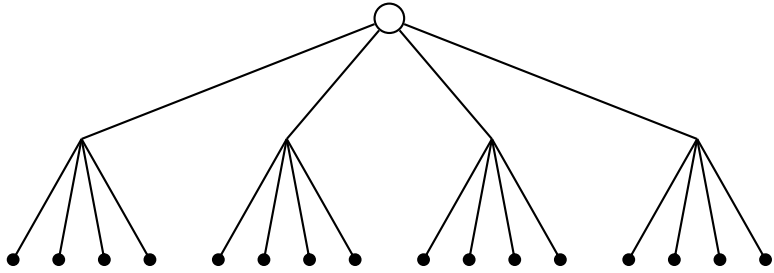

For our next theorem, which characterizes the inducibility of stars, we need Muirhead's Inequality \cite{hardy1952inequalities}: this inequality states that
$$\sum_{\pi \in S_d} \prod_{i = 1}^d x_{\pi(i)}^{l_i} \geq \sum_{\pi \in S_d} \prod_{i = 1}^d x_{\pi(i)}^{m_i}$$
for all nonnegative real numbers $x_1,x_2,\ldots,x_d$ whenever the vector $(l_1,l_2,\ldots,l_d)$ majorizes the vector $(m_1,m_2,\ldots,m_d)$. The sums on both sides are taken over all permutations of the indices $1,2,\ldots,d$.

\begin{theorem}\label{Cd}
For every fixed positive integer $d\geq 2$ and every $k\in \{2,3,\ldots,d\}$, the inducibility of the $k$-leaf star $C_k$ in strictly $d$-ary trees is
\begin{align*}
i_d(C_k)=\frac{d!}{(d-k)!(d^k-d)}\,.
\end{align*}
Moreover, $i_d(C_k)$ is an increasing sequence in $d$ for $k \geq 3$, starting with $d=k$.
\end{theorem}

\begin{proof}
Fix $d\geq 2$, and let us prove that for a strictly $d$-ary tree $T$ with $n$ leaves, we have
\begin{align*}
c(C_k,T)\leq \binom{d}{k} \cdot \frac{n^k-n}{d^k-d}
\end{align*}
for every $k\geq 2$, with equality if $T$ is a complete $d$-ary tree.

\medskip
For $k=2$, the assertion is immediate as $c(C_2,T)=\binom{|T|}{2}$. For $k\geq 3$, we use induction on $n$. The base cases $n<k$ are trivial as there cannot be any copies of $C_k$ in $T$. For the induction step, consider the $d$ branches $T_1,T_2,\ldots,T_d$ of $T$ with $n_1,n_2,\ldots,n_d$ leaves, respectively. So we have $n_1+n_2+\cdots+n_d=n$. We distinguish possible scenarios that can occur for a subset of $k$ leaves: 
\begin{itemize}
\item all $k$ leaves belong to the same branch of $T$. The total number of these subsets of leaves that induce $C_k$ is given by $c(C_k,T_1)+c(C_k,T_2)+\cdots + c(C_k,T_d)$,
\item one of the branches of $T$ contains more than one of the $k$ leaves, but not all of them. In this case the leaf-induced subtree is not isomorphic to $C_k$,
\item $k$ of the branches of $T$, say (without loss of generality) $T_1,T_2,\ldots,T_k$, contain exactly one of the leaves each. In this case, the $k$ leaves always induce $C_k$, yielding $n_1 \cdot n_2 \cdot \ldots \cdot n_k$ copies of $C_k$.
\end{itemize}
Therefore, we establish that
\begin{align*}
c(C_k,T)=\sum_{i=1}^d c(C_k,T_i)~+\sum_{\substack{\{i_1,i_2,\ldots,i_k\}\subseteq \{1,2,\ldots,d\}}}~\prod_{j=1}^k n_{i_j}\,.
\end{align*}
The induction hypothesis gives
\begin{equation}\label{eq:cCk}
c(C_k,T)\leq \frac{\binom{d}{k}}{d^k-d} \cdot \sum_{i=1}^d(n_i^k-n_i)~+\sum_{\substack{\{i_1,i_2,\ldots,i_k\}\subseteq \{1,2,\ldots,d\}}}~\prod_{j=1}^k n_{i_j}\,.
\end{equation}
On the other hand, the Multinomial Theorem yields the following decomposition:
\begin{align*}
n^k &=\Bigg(\sum_{i=1}^d n_i \Bigg)^k = \sum_{\substack{l_1,l_2,\ldots,l_d \geq 0 \\ l_1+l_2+ \cdots + l_d = k}} \binom{k}{l_1,l_2,\ldots,l_d} \prod_{i=1}^d n_i^{l_i} \\
&= \sum_{\substack{l_1,l_2,\ldots,l_d \geq 0 \\ l_1+l_2+ \cdots + l_d = k}} \binom{k}{l_1,l_2,\ldots,l_d} \frac{1}{d!} \sum_{\pi \in S_d} \prod_{i=1}^d n_{\pi(i)}^{l_i} \\
&= \sum_{i=1}^d n_i^k + \sum_{\substack{0 \leq l_1,l_2,\ldots,l_d < k \\ l_1+l_2+ \cdots + l_d = k}} \binom{k}{l_1,l_2,\ldots,l_d} \frac{1}{d!} \sum_{\pi \in S_d} \prod_{i=1}^d n_{\pi(i)}^{l_i}\,.
\end{align*}
Since every vector $(l_1,l_2,\ldots,l_d)$ of nonnegative integers with $l_1+l_2+\cdots+l_d = k$ majorizes the vector $(1,1,\ldots,1,0,0,\ldots,0)$ ($k$ ones, followed by $d-k$ zeros), we can apply Muirhead's Inequality to every term in the second sum of this decomposition:
$$\sum_{\pi \in S_d} \prod_{i=1}^d n_{\pi(i)}^{l_i} \geq \sum_{\pi \in S_d} \prod_{i=1}^k n_{\pi(i)} = k!(d-k)! \sum_{\substack{\{i_1,i_2,\ldots,i_k\}\subseteq \{1,2,\ldots,d\}}}~\prod_{j=1}^k n_{i_j}\,.$$
This gives us
\begin{align*}
n^k &\geq \sum_{i=1}^d n_i^k + \sum_{\substack{0 \leq l_1,l_2,\ldots,l_d < k \\ l_1+l_2+ \cdots + l_d = k}} \binom{k}{l_1,l_2,\ldots,l_d} \frac{k!(d-k)!}{d!} \sum_{\substack{\{i_1,i_2,\ldots,i_k\}\subseteq \{1,2,\ldots,d\}}}~\prod_{j=1}^k n_{i_j} \\
&= \sum_{i=1}^d n_i^k + (d^k-d) \frac{k!(d-k)!}{d!} \sum_{\substack{\{i_1,i_2,\ldots,i_k\}\subseteq \{1,2,\ldots,d\}}}~\prod_{j=1}^k n_{i_j}\,,
\end{align*}
using the Multinomial Theorem in the opposite direction now. We can rewrite this as
$$\sum_{\substack{\{i_1,i_2,\ldots,i_k\}\subseteq \{1,2,\ldots,d\}}}~\prod_{j=1}^k n_{i_j} \leq \frac{\binom{d}{k}}{d^k-d} \cdot \Big( n^k - \sum_{i=1}^d n_i^k \Big)\,.$$
Plugging this into~\eqref{eq:cCk} yields
\begin{align*}
c(C_k,T) &\leq \frac{\binom{d}{k}}{d^k-d} \cdot \sum_{i=1}^d(n_i^k-n_i) + \frac{\binom{d}{k}}{d^k-d} \cdot \Big( n^k - \sum_{i=1}^d n_i^k \Big) \\
&= \frac{\binom{d}{k}}{d^k-d} (n^k-n)\,,
\end{align*}
completing the induction. Furthermore, equality can only arise in this context if
$$n_1=n_2=\cdots =n_d\,.$$

Thus the inequality holds with equality if and only if, for every internal vertex $v$ of $T$, the number of leaves in the $d$ branches of the subtree of $T$ rooted at $v$ are the same: in this case, $T$ is a complete $d$-ary tree as well. 

The assertion on the inducibility follows by passing to the density and taking the limit:
\begin{align*}
i_d(C_k)=\lim_{n\to \infty} \Bigg\{\frac{\binom{d}{k}\cdot \big(n^k-n \big)/(d^{k}-d)}{\binom{n}{k}}\Bigg\}=\frac{d!}{(d-k)!\cdot (d^k-d)}\,.
\end{align*}

\medskip
We now turn to the second assertion of the theorem.

\medskip

\textbf{Claim}: For every given positive integer $k\geq 3$,
\begin{align*}
u_k(x)= \frac{(x-1)(x-2)\cdots (x-k+1)}{x^{k-1}-1}
\end{align*}
is an increasing function for $x\geq k$. 

\medskip
\textit{Proof of the Claim}: Indeed, we have
\begin{align*}
\log u_k(x)=-\log (x^{k-1}-1)+ \sum_{i=1}^{k-1} \log (x-i) 
\end{align*}
so that
\begin{align*}
\frac{d}{dx}\Big(\log u_k(x) \Big)=&-\frac{k-1}{x^{k-1}-1}\cdot x^{k-2}+\sum_{i=1}^{k-1}\frac{1}{x-i}\\
&=\sum_{i=1}^{k-1}\Bigg(\frac{1}{x-i}-\frac{x^{k-2}}{x^{k-1}-1}\Bigg)\\
&=\frac{1}{x^{k-1}-1}\cdot \sum_{i=1}^{k-1}\Bigg(\frac{-1+i\cdot x^{k-2}}{x-i}\Bigg)>0
\end{align*}
as $x\geq k\geq 3$.

\medskip
We conclude that $\big(i_d(C_k)\big)_{d\geq k}$ is an increasing sequence as soon as $k\geq 3$.
\end{proof}

\bigskip
By a \textit{$d$-ary caterpillar}, we mean a strictly $d$-ary tree in which every non-leaf vertex has $d-1$ children that are leaves, except for the lowest, which has $d$ children that are leaves. Note that the non-leaf vertices must lie on a single path starting at the root. We denote the $d$-ary caterpillar with $k$ leaves by $F^d_k$ (refer to Fig.~\ref{Ternary caterpillars.} for ternary caterpillars).

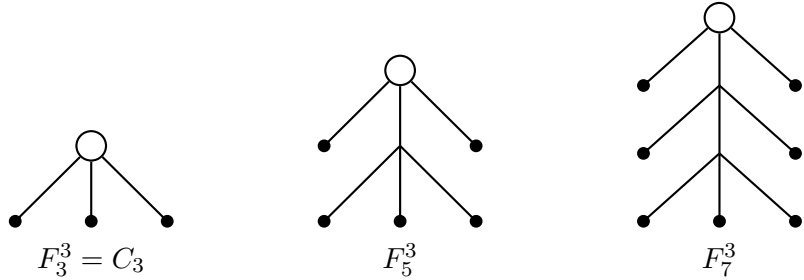
\begin{figure}[htbp]\centering
  \begin{subfigure}[b]{0.2\textwidth} \centering  
\begin{tikzpicture}[thick,level distance=10mm]
\tikzstyle{level 1}=[sibling distance=10mm]
\node [circle,draw]{}
child {[fill] circle (2pt)}
child {[fill] circle (2pt)}
child {[fill] circle (2pt)};
\end{tikzpicture}
   \caption{$F^3_3=C_3$}
  \end{subfigure}\qquad
\begin{subfigure}[b]{0.2\textwidth} \centering 
  \begin{tikzpicture}[thick,level distance=10mm]
\tikzstyle{level 1}=[sibling distance=10mm]
\tikzstyle{level 2}=[sibling distance=10mm]
\node [circle,draw]{}
child {[fill] circle (2pt)}
child {child {[fill] circle (2pt)}child {[fill] circle (2pt)}child {[fill] circle (2pt)}}
child {[fill] circle (2pt)};
\end{tikzpicture}
   \caption{$F^3_5$}
  \end{subfigure} \qquad
\begin{subfigure}[b]{0.2\textwidth} \centering 
  \begin{tikzpicture}[thick,level distance=9mm]
\tikzstyle{level 1}=[sibling distance=10mm]
\tikzstyle{level 2}=[sibling distance=10mm]
\node [circle,draw]{}
child {[fill] circle (2pt)}
child {child {[fill] circle (2pt)}child {child {[fill] circle (2pt)}child {[fill] circle (2pt)}child {[fill] circle (2pt)}}child {[fill] circle (2pt)}}
child {[fill] circle (2pt)};
\end{tikzpicture}
   \caption{$F^3_7$}
  \end{subfigure} 
	\caption{Ternary caterpillars $F^3_k$.} \label{Ternary caterpillars.}
\end{figure}

\begin{theorem}\label{inducibility of the binary caterpillar is 1 in d ary trees}
Let $d\geq 2$ be a fixed positive integer. For the binary caterpillar $F^2_k$, we have
\begin{align*}
\max_{\substack{|T|=n \\T~\text{\rm strictly $d$-ary tree}}}\frac{c(F^2_k,T)}{\binom{n}{k}}=1- \mathcal{O}(n^{-1})
\end{align*}
for every $k$ and every $d$. In particular, $i_d(F^2_k)=I_d(F^2_k)=1$ for every $k$ and every $d$.
\end{theorem}

\begin{proof}
Fix $d\geq 2$. We use a direct counting argument which gives us a lower bound on the number of copies of $F^2_k$ in $F^d_n$. A binary caterpillar can be constructed by attaching exactly one pendant edge to all vertices of a path, except the last one (that is furthest away from the root).

Thus, a copy of $F^2_k$ in $F^d_n$ can be obtained by first choosing a subset of $k$ vertices from the set of internal vertices of $F^d_n$, and then for each of them, choosing one of its (at least) $d-1$ children that are leaves. Since every strictly $d$-ary tree $T$ has exactly $(|T|-1)/(d-1)$ internal vertices, we deduce that there are at least 
\begin{align*}
\binom{\frac{n-1}{d-1}}{k} \cdot (d-1)^k
\end{align*}
copies of $F^2_k$ in $F^d_n$. Therefore, we get
\begin{align*}
c(F^2_k,F^d_n) \geq \binom{\frac{n-1}{d-1}}{k} \cdot (d-1)^k = \frac{n^k}{k!} - \mathcal{O}(n^{k-1})\,.
\end{align*}
Hence, because we have $c(F^2_k,F^d_n) \leq  \binom{n}{k} \leq n^k/k!$ by definition, the assertion on the maximum density of $F^2_k$ in strictly $d$-ary trees and thus the inducibility $i_d(F^2_k)$ follows. In particular, we obtain $i_d(F^2_k)=I_d(F^2_k)=1$.
\end{proof} 
It can actually be shown that 
$$c(F^2_k,F^d_n) = (d-1)^{k-1} \binom{\frac{n-1}{d-1}}{k-1} \cdot \frac{2n-(d-2)(k-2)}{2k}$$
for $k,n > 1$, but this precision becomes immaterial when computing the inducibility $i_d(F^2_k)$.


\section{Inducibility vs. Maximum Density}\label{IndVsMax}

Our first goal in this section is to prove a conjecture from~\cite{czabarka2016inducibility}, which states that the maximum of $\gamma(D,T)$ over trees $T$ with $n$ leaves converges fairly quickly to the inducibility $I_d(D)$, namely at a rate of $n^{-1}$. (To be precise, the conjecture was only made in the binary case.)

\begin{theorem}\label{maxdensityId}
For every fixed positive integer $d\geq 2$ and every $d$-ary tree $D$, we have
\begin{align*}
\max_{\substack{|T|=n\\T~\text{\rm $d$-ary tree}}} \gamma(D,T)= I_d(D)+ \mathcal{O}(n^{-1})
\end{align*}
for all $n$. In particular, 
\begin{align*}
\lim_{n\to \infty} \max_{\substack{|T|=n\\T~\text{\rm $d$-ary tree}}} \gamma(D,T) =I_d(D)\, ,
\end{align*}
where the limit is that of a decreasing sequence.
\end{theorem}

\begin{proof}
Fix $d\geq 2$. Let $D$ and $T$ be arbitrary $d$-ary trees such that $|D|\leq |T|$. Since $c(D,T)$ represents the number of subsets of $|D|$ leaves of $T$ that induce a copy of $D$, we immediately deduce that the $|T|$ leaves of $T$ are contained in $|D|\cdot c(D,T)/|T|$ copies of $D$ on average. Thus, there exist leaves $l_1$ and $l_2$ of $T$ satisfying the relation
\begin{align}\label{InView}
c_{l_1}(D,T) \leq \frac{|D| \cdot c(D,T)}{|T|} \leq c_{l_2}(D,T) 
\end{align}
where $c_{l}(D,T)$ stands for the number of $l$-containing subsets of $|D|$ leaves of $T$ that induce a copy of $D$. 
\begin{itemize}
\item We can assume $|T|\geq 2$. The number of subsets of $|D|$ leaves of $T$ not involving the leaf $l_1$ that induce a tree isomorphic to $D$ is exactly $c(D,T) - c_{l_1}(D,T)$, which is therefore at least $\big(1-\frac{|D|}{|T|} \big)c(D,T)$ in view of relation~\eqref{InView}. 
Create from $T$ a new tree $T'$ by deleting the leaf $l_1$ (and suppress its former neighbor, if its new degree is 2).
It follows that
\begin{align*}
\max_{|T'|=n-1} c(D,T') \geq \Big(1-\frac{|D|}{n} \Big)\max_{|T|=n} c(D,T)
\end{align*}
as $T$ was taken to be an arbitrary $d$-ary tree. Hence, by passing to the density by dividing by $\binom{n-1}{|D|}$, we obtain
\begin{align*}
\max_{\substack{|T'|=n-1\\T'~\text{$d$-ary tree}}} \gamma(D,T') \geq \max_{\substack{|T|=n\\T~\text{$d$-ary tree}}} \gamma(D,T) 
\end{align*}
for every $n\geq 2$. So the sequence 
\begin{align*}
\Biggl( \max_{\substack{|T|=n\\T~\text{$d$-ary tree}}} \gamma(D,T) \Biggl)_{n\geq 1}
\end{align*}
is decreasing and bounded from below, which means that the limit
\begin{align*}
\lim_{n\to \infty}\max_{\substack{|T|=n\\T~\text{$d$-ary tree}}} \gamma(D,T)
\end{align*}
exists and is $I_d(D)$.

\item Denote by $T^{+}$ the $d$-ary tree obtained by replacing the leaf $l_2$ by an internal vertex with two leaves $l_2,l_2^{\prime}$ attached to it. That way, the number of copies of $D$ in $T^{+}$ not involving $l_2^{\prime}$ is just $c(D,T)$, whereas the number of copies of $D$ in $T^{+}$ involving $l_2^{\prime}$ is no less than the number of copies of $D$ in $T$ involving $l_2$. Therefore, by relation~\eqref{InView}, the quantity $\big(1+\frac{|D|}{|T|} \big)c(D,T)$ offers a natural lower bound on $c(D,T^{+})$. It follows that
\begin{align*}
\Big(1+\frac{|D|}{n} \Big)\max_{|T|=n} c(D,T) \leq  \max_{|T|=n+1} c(D,T)
\end{align*}
as $T$ was assumed to be an arbitrary $n$-leaf $d$-ary tree. Consequently, passing to the density, we obtain
\begin{align*}
\max_{\substack{|T|=n+1\\T~\text{$d$-ary tree}}} \gamma(D,T) \geq \Bigg(1-\frac{|D|(-1+|D|)}{n(n+1)} \Bigg)\max_{\substack{|T|=n\\T~\text{$d$-ary tree}}} \gamma(D,T)
\end{align*}
for every $n$, and by $p$-fold iteration
\begin{align*}
\max_{\substack{|T|=n+p\\T~\text{$d$-ary tree}}} \gamma(D,T) \geq  \max_{\substack{|T|=n\\T~\text{$d$-ary tree}}} \gamma(D,T)\cdot \prod_{j=0}^{p-1}\Bigg(1- \frac{|D|(-1+|D|)}{(n+p-j)(n+p-j-1)}\Bigg)
\end{align*}
for all $n,p$ with $p\geq 1$ and $n\geq |D|$. Fixing $n\geq 2$ and $p \geq 1$, we have 
\begin{align*}
0 \leq \frac{|D|(-1+|D|)}{(n+p-j)(n+p-j-1)} <1
\end{align*}
for every $0\leq j \leq p-1$. A simple induction on $p$ yields
\begin{align*}
\prod_{j=0}^{p-1}\Bigg(1- \frac{|D|(-1+|D|)}{(n+p-j)(n+p-j-1)}\Bigg) \geq 1-\sum_{j=0}^{p-1} \frac{|D|(-1+|D|) }{(n+p-j)(n+p-j-1)}\,.
\end{align*}

Now, letting $p \to \infty$ instantly gives the estimate
\begin{align*}
I_d(D) \geq \max_{\substack{|T|=n\\T~\text{$d$-ary tree}}} \gamma(D,T) \cdot \Bigg(1-\sum_{i=0}^{\infty} \frac{|D|(-1+|D|) }{(n+i+1)(n+i)} \Bigg)
\end{align*}
for every $n$. Since
\begin{align*}
\sum_{i=0}^{\infty} \frac{1}{(n+i+1)(n+i)}= \sum_{i=0}^{\infty} \Big( \frac{1}{n+i} - \frac{1}{n+i+1}\Big) = \frac{1}{n}\,,
\end{align*}
this shows that
\begin{align*}
I_d(D) \geq \Big(1- \frac{|D|(-1+|D|)}{n} \Big)\max_{\substack{|T|=n\\T~\text{$d$-ary tree}}} \gamma(D,T)\,.
\end{align*}
\end{itemize}

Now we combine the two contributions to obtain
\begin{align*}
0\leq \max_{\substack{|T|=n\\T~\text{$d$-ary tree}}}  \gamma(D,T) - I_d(D) \leq \frac{|D|(-1+|D|)}{n} 
\end{align*}
for every $n$. The desired asymptotic formula follows immediately.
\end{proof}

We remark that the averaging reasoning employed in the proof of Theorem~\ref{maxdensityId} does not work for a strictly $d$-ary tree. For example, removing one leaf (as well as the single edge incident to it) from a strictly $d$-ary tree yields a tree that is no longer strictly $d$-ary as soon as $d\geq 3$.
Moreover, the error term $\mathcal{O}(n^{-1})$ is generally best possible, as e.g.~the discussion of the complete binary tree of height $2$ (with four leaves) in~\cite{czabarka2016inducibility} shows.

\medskip

Our next focus is to prove that $i_d(D)$ and $I_d(D)$ are always equal for every $d$-ary tree $D$ and every $d$.

\begin{lemma}\label{lemForTpvsT}
For a fixed positive integer $k$, we have
\begin{align*}
\frac{\binom{n}{k}}{\binom{p+n}{k}} = 1-\mathcal{O}(p/n)
\end{align*}
as $p\geq 1$ and $n/p \to \infty$.
\end{lemma}

\begin{proof}
We have
\begin{align*}
\frac{\binom{n}{k}}{\binom{p+n}{k}} &= \frac{n^k \Big(1 - \mathcal{O}\big(\frac{1}{n}\big)\Big) }{(p+n)^k \Big(1 -  \mathcal{O}\big(\frac{1}{p+n} \big) \Big)}\\
&=\Big(1+\frac{p}{n}\Big)^{-k}\Bigg(1 - \mathcal{O}\Big(\frac{1}{n}\Big)\Bigg)\\
&=1- \mathcal{O}\Big(\frac{p}{n}\Big)
\end{align*}
as $p\geq 1$ and $n/p \to \infty$.
\end{proof}

\begin{theorem}\label{towards idId}
Fix a positive integer $d\geq 2$, and let $D$ be a $d$-ary tree. For every positive integer $n \equiv 1 \mod (d-1)$ and every $d$-ary tree $T$ with $\lfloor n^{\frac{1}{2}} \rfloor$ leaves, there exists a strictly $d$-ary tree $T^*$ with $n$ leaves such that the asymptotic formula
\begin{align*}
\gamma(D,T) =  \gamma(D,T^{*}) + \mathcal{O}(n^{-1/2})
\end{align*}
holds as $n\rightarrow\infty$, and the $\mathcal{O}$-constant depends on $d$ only.

In particular, we have
\begin{align*}
i_d(D)=I_d(D)\,.
\end{align*}
\end{theorem}

\begin{proof}
Fix $d\geq 2$ and $n \equiv 1 \mod (d-1)$. Consider an arbitrary $d$-ary tree $T$ (not necessarily a strictly $d$-ary tree) such that
$|T|= \lfloor n^{\frac{1}{2}} \rfloor = n^{\frac{1}{2}} - \mathcal{O}(1)\,.$

We describe an explicit construction for $T^*$; the line of reasoning follows probabilistic ideas.

\begin{itemize}
\item For every $r \in\{2,3,\ldots,d-1\}$, add $d-r$ more branches of one leaf each to every internal vertex of $T$ whose number of children is $r$. Call the augmented tree $T^{\prime}$, and denote by $\tilde{L}(T^{\prime})$ the set of the additional leaves added to $T$ to obtain the strictly $d$-ary tree $T^{\prime}$. If $|T|_r$ stands for the number of internal vertices of $T$ with $r$ children, then we have
\begin{align*}
|\tilde{L}(T^{\prime})| = \sum_{r=2}^{d-1} (d-r)|T|_r <(d-2) |T|,~|\tilde{L}(T^{\prime})| =\mathcal{O}\big(n^{\frac{1}{2}}\big)
\end{align*}
because it is easy to see that the total number of internal vertices of $T$ is less than its number of leaves. Note that 
\begin{align*}
|T^{\prime}| =|T|+|\tilde{L}(T^{\prime})| = \Theta \big(n^{\frac{1}{2}} \big)\,.
\end{align*}

\item Consider the tree $T^{\prime}$. Let $m$ be the greatest positive integer satisfying both $m \equiv 1 \mod (d-1)$ and $m\leq n^{\frac{1}{2}}$. So we have $\lfloor n^{\frac{1}{2}} \rfloor - (d-2) \leq m \leq \lfloor n^{\frac{1}{2}} \rfloor $. Since it suffices to prove the statement for sufficiently large $n$, we may assume that $m\geq d$. Choose an arbitrary strictly $d$-ary tree $S$ with $m-(d-1)$ leaves so that
\begin{align*}
\lfloor n^{\frac{1}{2}} \rfloor +3 -2\cdot d \leq |S| \leq \lfloor n^{\frac{1}{2}} \rfloor +1 -d\,.
\end{align*}

Append a copy of $S$ to every leaf $l$ of $T^{\prime}$ that does not belong to $\tilde{L}(T^{\prime})$ by identifying its root with $l$. Call the resulting tree $T^{\prime \prime}$. We shall refer to the tree $S$ as a `dangling' tree of $T^{\prime \prime}$. Note that 
\begin{align*}
|T^{\prime \prime}|=|T|\cdot |S|+|\tilde{L}(T^{\prime})| = n - \mathcal{O}\big(n^{\frac{1}{2}} \big)
\end{align*}
and that $|T^{\prime \prime}| < n$ in view of the inequalities
\begin{align*}
|T| \leq n^{\frac{1}{2}},~ |S|\leq n^{\frac{1}{2}} -(d-1),~|\tilde{L}(T^{\prime})|<(d-1)|T|\,.
\end{align*}

\item  Additionally, pick an arbitrary strictly $d$-ary tree $S_{P}$ with
\begin{align*}
1+n-|T^{\prime \prime}|=\mathcal{O}\big(n^{\frac{1}{2}} \big)
\end{align*}
leaves and append the root of $T^{\prime \prime}$ to a leaf of $S_P$. Denote by $T^*$ the strictly $d$-ary tree that results from this construction. Note that
\begin{align*}
|T^*|=|T^{\prime \prime}|+|S_P|-1=n\,.
\end{align*}
\end{itemize}

\medskip
Now, let $1 \leq k \leq |T|$ be an arbitrary but fixed positive integer, and pick $k$ leaves of $T^*$ uniformly at random.

The probability $P$ that none of the $k$ randomly chosen leaves of $T^*$ lies in $\tilde{L}(T^{\prime})$ or $S_P$ and no two of them belong to the same dangling tree of $T^*$ is exactly
 \begin{align*}
\frac{\binom{|T|}{k}}{\binom{|T^{*}| }{k}}\cdot |S|^k\,.
\end{align*}
In words: since there are exactly $|T|$ dangling trees in $T^*$, we choose $k$ of them and one leaf from each of the $k$ chosen dangling trees to obtain such a subset of $k$ leaves of $T^*$. 

With Lemma~\ref{lemForTpvsT} at our disposal, we obtain

\begin{align}\label{choosekleavesA}
\begin{split}
\frac{\binom{|T|}{k}}{\binom{|T^{*}| }{k}}\cdot |S|^k &=\frac{\big(|T|\cdot |S|\big)^k\Big(1 - \mathcal{O}\big(|T|^{-1} \big)\Big)}{n^k \Big(1 - \mathcal{O}\big( n^{-1}\big)\Big)}\\
&=1- \mathcal{O}\big( n^{-\frac{1}{2}}\big)
\end{split}
\end{align}
as $n \to \infty$.

\medskip

Now let $D$ be a $d$-ary tree with $k$ leaves. Note that the tree induced by $k$ leaves of $T^*$ that belong to $k$ distinct dangling trees is equal to the tree induced by the $k$ leaves of $T$ to which these $k$ dangling tree were attached. Hence the probability that $k$ randomly chosen leaves of $T^*$ are from distinct dangling trees of $T^*$ and induce a copy of $D$ is given by 
\begin{align*}
\frac{c(D,T)}{\binom{n}{k}} \cdot |S|^k
\end{align*}
(recall that $T^*$ has $n$ leaves). From this observation and the fact that $\gamma(D,T^*)$ is exactly the probability that $k$ distinct randomly chosen leaves of $T^*$ induce a copy of $D$, we deduce that
\begin{align*}
\gamma(D,T^*) = \frac{c(D,T)}{\binom{n}{k}} \cdot |S|^k + Q \cdot \Bigg(1- \frac{\binom{|T|}{k} \cdot |S|^k}{\binom{n}{k}} \Bigg)
\end{align*}
by virtue of the law of total probability, where $Q$ stands for the probability that $k$ distinct leaves of $T^*$ induce a copy of $D$ under the condition that the event `$k$ randomly chosen leaves of $T^*$ are from distinct dangling trees of $T^*$' has not occurred. 

\medskip
This implies that
\begin{align*}
\gamma(D,T^*) = \frac{c(D,T)}{\binom{|T|}{k}} + \Bigg(1- \frac{\binom{|T|}{k} \cdot |S|^k}{\binom{n}{k}} \Bigg) \cdot \Bigg(Q- \frac{c(D,T)}{\binom{|T|}{k}}\Bigg)\,,
\end{align*}
and since $Q$ and $\gamma(D,T) = c(D,T)/\binom{|T|}{k}$ are both between $0$ and $1$, it follows from the asymptotic formula~\eqref{choosekleavesA} that 
\begin{align*}
\gamma(D,T^*) - \gamma(D,T) = \mathcal{O}(n^{-1/2})\,.
\end{align*}
This finishes the proof of the first part of the theorem.

\medskip
Finally, the immediate consequence we obtain is that
\begin{align*}
I_d(D)=\limsup_{\substack{|T| \to \infty \\ T~\text{$d$-ary tree}}} \gamma(D,T) = \limsup_{|T^*| \to \infty} \gamma(D,T^*) \leq \limsup_{\substack{|T| \to \infty \\ T~\text{strictly $d$-ary tree}}} \gamma(D,T) = i_d(D)\,.
\end{align*}
In other words, this shows that $I_d(D)\leq i_d(D)$. Thus, the proof of the second part of the theorem is completed as well because we have $I_d(D)\geq i_d(D)$ by definition.
\end{proof}

With Theorem~\ref{towards idId} and its proof at hand, we can now prove an analogue of Theorem~\ref{maxdensityId} for the maximum density in strictly $d$-ary trees. 

\begin{corollary}\label{lowerBoundAsymptid}
For every fixed positive integer $d\geq 2$ and every $d$-ary tree $D$, we have
\begin{align*}
\max_{\substack{|T|=n\\T~\text{\rm strictly $d$-ary tree}}} \gamma(D,T) = i_d(D)+ \mathcal{O}(n^{-1/2})
\end{align*}
for every $n$. In particular, we have
\begin{align*}
\lim_{n \to \infty} \max_{\substack{|T|=n\\T~\text{\rm strictly $d$-ary tree}}} \gamma(D,T) = i_d(D)\,.
\end{align*}
\end{corollary}

\begin{proof} 
Using the identity $i_d(D)=I_d(D)$ from Theorem~\ref{towards idId} together with the first claim of Theorem~\ref{maxdensityId}, namely
\begin{align*}
\max_{\substack{|T|=n\\T~\text{strictly $d$-ary tree}}} \gamma(D,T)\leq \max_{\substack{|T|=n\\T~\text{$d$-ary tree}}} \gamma(D,T) \leq i_d(D)+\mathcal{O}\big(n^{-1}\big)\,,
\end{align*}
immediately yields 
\begin{align} \label{also}
\max_{\substack{|T|=n\\T~\text{strictly $d$-ary tree}}} \gamma(D,T) \leq i_d(D) + \mathcal{O}\big(n^{-1}\big)\,.
\end{align}
Also, the second claim of Theorem~\ref{maxdensityId} gives
\begin{align} \label{egy}
i_d(D)=I_d(D)\leq \max_{\substack{|T|=\lfloor n^{1/2}\rfloor \\T~\text{$d$-ary tree}}}  \gamma(D,T)\, ,
\end{align}
while the first claim of Theorem~\ref{towards idId} guarantees a particular strictly $d$-ary tree $T^*$ on $n$ leaves for the maximizer tree
in \eqref{egy}, such that
\begin{align} \label{ketto}
\gamma(D,T)=\gamma(D,T^*)+\mathcal{O}\big(n^{-\frac{1}{2}}\big)\,.
\end{align}
Formulae \eqref{egy} and \eqref{ketto} immediately give 
\begin{align*}
i_d(D)\leq \max_{\substack{|T'|=n\\T'~\text{strictly $d$-ary tree}}} \gamma(D,T') +\mathcal{O}\big(n^{-\frac{1}{2}}\big)\, ,
\end{align*}
which, together with \eqref{also}, completes the proof of the Corollary.
\end{proof}

\bigskip
In certain cases, we suspect a stronger asymptotic result on the maximum density in strictly $d$-ary trees \cite{ClassDossouOloryWagner}. This, in particular, happens when
\begin{align*}
\displaystyle \max_{\substack{|T|=n \\ T~\text{$d$-ary tree}}} c(D,T) 
\end{align*}
is attained by strictly $d$-ary trees for every $n\equiv 1 \mod (d-1)$. Finally, a natural question to ask at this point is the following:

\begin{question}
Can the $\mathcal{O}$-term in Corollary~\ref{lowerBoundAsymptid} be improved somewhat further for general $d$-ary trees $D$?
\end{question}

\section{Inducibility of a Binary Tree in $d$-ary Trees} \label{bin}

Our aim in this section is to compare the inducibilities of a binary tree $B$ in $d$-ary trees, for different values of $d$. It turns out that $i_d(B)$ is independent of $d$.

\begin{theorem}\label{Upper bound for all binary trees}
Every binary tree $B$ satisfies
\begin{align*}
i_d(B)= i_2(B)=I_d(B)=I_2(B)
\end{align*}
for every $d$.
\end{theorem}

\begin{proof}
The following correspondence is our stepping stone for proving the assertion. Let $d\geq 3$ be fixed, and fix an arbitrary total order $\prec$ on the set of strictly $d$-ary trees. For a strictly $d$-ary tree $T$, we always order the branches $T_1,T_2,\ldots,T_d$ in such a way that
\begin{align*}
T_1 \preceq T_2 \preceq \cdots \preceq T_d\,.
\end{align*}

From the tree $T$, we build a binary tree $G(T)$ by means of the recursive algorithm depicted in Fig.~\ref{GT from T}, starting with the single leaf being invariant under $G$. More specifically, for $|T|>1$, the tree $G(T)$ is obtained as follows:
\begin{itemize}
\item draw a path on $d-1$ vertices $v_2,v_3,\ldots,v_d$ in this order ($v_2$ and $v_d$ are the endvertices of the path);
\item attach a leaf $l_i$ (by dropping a pendant edge) to every vertex $v_i$ of the path except for the lowest vertex $v_2$;
\item replace the vertex $v_2$ of the path by an internal vertex with two leaves $l_1$ and $l_2$ attached to it;
\item append the root of $G(T_i)$ to leaf $l_i$ for every $i \in \{1,2,\ldots,d\}$. The vertex $v_d$ is the root of $G(T)$.
\end{itemize}

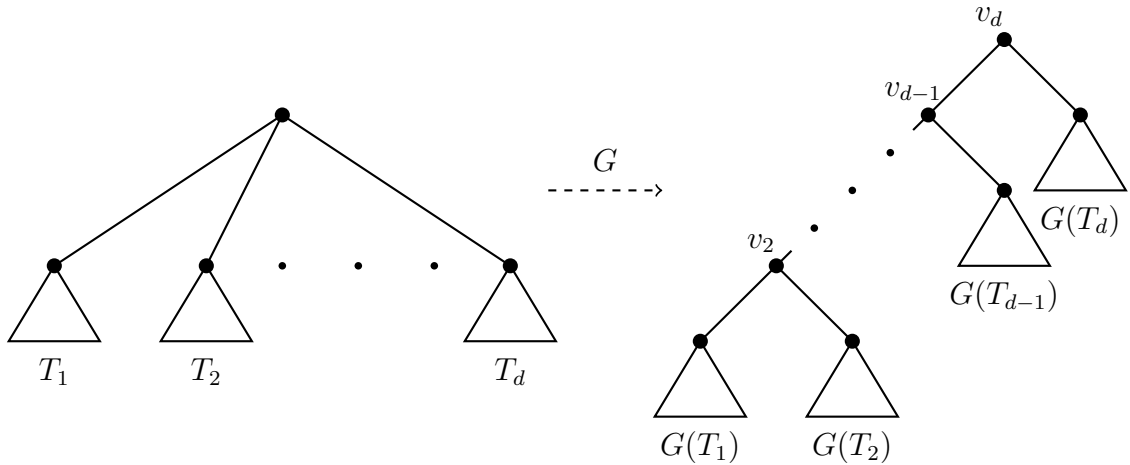
\begin{figure}[htbp]\centering
\begin{tikzpicture}[thick]

\node[fill=black,circle,inner sep=2pt] at (0,0) {};
\node[fill=black,circle,inner sep=2pt] at (2,0) {};
\node[fill=black,circle,inner sep=2pt] at (6,0) {};

\node[fill=black,circle,inner sep=1pt] at (3,0) {};
\node[fill=black,circle,inner sep=1pt] at (4,0) {};
\node[fill=black,circle,inner sep=1pt] at (5,0) {};

\node[fill=black,circle,inner sep=2pt] at (3,2) {};

\draw (-0.6,-1)--(0.6,-1)--(0,0)--cycle;
\draw (1.4,-1)--(2.6,-1)--(2,0)--cycle;
\draw (5.4,-1)--(6.6,-1)--(6,0)--cycle;

\draw (0,0)--(3,2);
\draw (2,0)--(3,2);
\draw (6,0)--(3,2);

\node at (0,-1.4) {$T_1$};
\node at (2,-1.4) {$T_2$};
\node at (6,-1.4) {$T_d$};

\node at (7.25,1.4) {$G$};
\draw [dashed,->] (6.5,1)--(8,1);

\node[fill=black,circle,inner sep=2pt] at (8.5,-1) {};
\node[fill=black,circle,inner sep=2pt] at (9.5,0) {};
\node[fill=black,circle,inner sep=2pt] at (11.5,2) {};
\node[fill=black,circle,inner sep=2pt] at (12.5,3) {};

\node[fill=black,circle,inner sep=2pt] at (10.5,-1) {};
\node[fill=black,circle,inner sep=2pt] at (12.5,1) {};
\node[fill=black,circle,inner sep=2pt] at (13.5,2) {};

\node[fill=black,circle,inner sep=1pt] at (10,0.5) {};
\node[fill=black,circle,inner sep=1pt] at (10.5,1) {};
\node[fill=black,circle,inner sep=1pt] at (11,1.5) {};

\draw (7.9,-2)--(9.1,-2)--(8.5,-1)--cycle;
\draw (9.9,-2)--(11.1,-2)--(10.5,-1)--cycle;
\draw (11.9,0)--(13.1,0)--(12.5,1)--cycle;
\draw (12.9,1)--(14.1,1)--(13.5,2)--cycle;

\draw (8.5,-1)--(9.7,0.2);
\draw (11.3,1.8)--(12.5,3);

\draw (9.5,0)--(10.5,-1);
\draw (11.5,2)--(12.5,1);
\draw (12.5,3)--(13.5,2);

\node at (9.3,0.3) {$v_2$};
\node at (11.3,2.3) {$v_{d-1}$};
\node at (12.3,3.3) {$v_d$};

\node at (8.5,-2.4) {$G(T_1)$};
\node at (10.5,-2.4) {$G(T_2)$};
\node at (12.5,-0.4) {$G(T_{d-1})$};
\node at (13.5,0.6) {$G(T_d)$};

\end{tikzpicture}
\caption{$T$ is a strictly $d$-ary tree and $G(T)$ is its corresponding binary tree under the tree-transformation $G$ described in the proof of Theorem~\ref{Upper bound for all binary trees}.}\label{GT from T}
\end{figure}
For example, $G$ maps the star $C_d$ to the binary caterpillar $F^2_d$. 

\medskip
With this transformation at hand, let us prove that the inequality
\begin{align*}
c\big(B,G(T)\big)\geq c(B,T)
\end{align*}
holds for every binary tree $B$ and every strictly $d$-ary tree $T$.

\medskip
The construction of $G(T)$ also yields a natural bijection between the leaves of $T$ and those of $G(T)$. We can show by induction that if a set of leaves of $T$ induces a copy of a binary tree $B$, then so do the corresponding leaves in $G(T)$. This is certainly true when $|T|=1$. For the induction step, we have two cases:
\begin{itemize}
\item If the leaves that induce $B$ all lie in one branch $T_i$ for some $i \in \{1,2,\ldots,d\}$, then the corresponding leaves lie in $G(T_i)$, and we are done by the induction hypothesis.
\item Otherwise, they lie in exactly two branches $T_i$ and $T_j$ for some $i \neq j \in \{1,2,\ldots,d\}$, and the two leaf sets induce the two branches of $B$. By the induction hypothesis, this is also true for the corresponding leaves in $G(T_i)$ and $G(T_j)$, and we are done again.
\end{itemize}

This shows that every subset of leaves of $T$ that induces a copy of $B$ corresponds to a unique subset of leaves of $G(T)$ that induces a copy of $B$. Therefore, there is an injection from the copies of $B$ in $T$ to the copies of $B$ in $G(T)$, and we have $c(B,T) \leq c\big(B,G(T)\big)$ as claimed.

Consequently, we arrive at
\begin{align*}
i_d(B)\leq \limsup_{\substack{|T|\to \infty \\ T~\text{strictly $d$-ary tree}}} \gamma\big(B,G(T)\big) \leq \limsup_{\substack{|T|\to \infty \\ T~\text{binary tree}}} \gamma\big(B,T\big)  = i_2(B)\,.
\end{align*}
On the other hand, we have both $i_2(B)=I_2(B)$ and $I_d(B)\geq I_2(B)$ by definition, while Corollary~\ref{lowerBoundAsymptid} gives us $i_d(B)=I_d(B)$. Hence, the statement of the theorem follows.
\end{proof}

We point out that the analogue of Theorem~\ref{Upper bound for all binary trees} for non-binary trees is not true in general: as Theorem~\ref{Cd} shows, the inducibility $i_d(C_k)$ of the $k$-leaf star is a strictly increasing function of $d$ for $k\geq 3$.

From here onwards, we shall use only $i_d(D)$. Our next section deals with some bounds on the inducibility.

\section{Some General Results} \label{generic}

Let us say something about how small $i_d(D)$ can be:

\begin{proposition}\label{howsamllJS}
Let $d\geq 2$ be an arbitrary but fixed positive integer. Every $d$-ary tree $D$ with at least two leaves satisfies
\begin{align*}
i_d(D) \geq \frac{(-1+|D|)!}{-1+|D|^{-1+|D|}}\,.
\end{align*}
In particular, every $d$-ary tree has positive inducibility.
\end{proposition}

\begin{proof}
We employ a tree-construction similar to the one used for proving that $i_d(D)=I_d(D)$. Fix $d\geq 2$. For two strictly $d$-ary trees $S_1,S_2$, denote by $\mathcal{F}(S_1;S_2)$ the unique strictly $d$-ary tree that is formed by appending the root of $S_2$ to every leaf of $S_1$. See Fig.~\ref{TreeFS1S2} for a picture. So we have $|\mathcal{F}(S_1;S_2)|=|S_1|\cdot |S_2|$.

\begin{figure}[htbp]\centering
\begin{tikzpicture}[thick]

\node[fill=black,circle,inner sep=2pt] at (0,0) {};
\node[fill=black,circle,inner sep=2pt] at (4,0) {};
\node[fill=black,circle,inner sep=2pt] at (2,2) {};

\node[fill=black,circle,inner sep=2pt] at (-1,-1.73) {};
\node[fill=black,circle,inner sep=2pt] at (1,-1.73) {};
\node[fill=black,circle,inner sep=2pt] at (3,-1.73) {};
\node[fill=black,circle,inner sep=2pt] at (5,-1.73) {};

\node[fill=black,circle,inner sep=1pt] at (1.4,0) {};
\node[fill=black,circle,inner sep=1pt] at (1.8,0) {};
\node[fill=black,circle,inner sep=1pt] at (2.2,0) {};
\node[fill=black,circle,inner sep=1pt] at (2.6,0) {};

\node at (2,0.8) {$S_1$};
\node at (0,-1.2) {$S_2$};
\node at (4,-1.2) {$S_2$};

\draw (1.4,0)--(0,0)--(2,2)--(4,0)--(2.6,0);

\draw (-1,-1.73)--(1,-1.73)--(0,0)--cycle;
\draw (3,-1.73)--(5,-1.73)--(4,0)--cycle;

\node at (2,-2.2) {$\mathcal{F}(S_1;S_2)$};

\node at (-0.2,0.3) {$1$};
\node at (4.4,0.3) {$|S_1|$};

\end{tikzpicture}
\caption{A rough picture of the tree $\mathcal{F}(S_1;S_2)$ defined in the proof of Proposition~\ref{howsamllJS}.} \label{TreeFS1S2}
\end{figure}
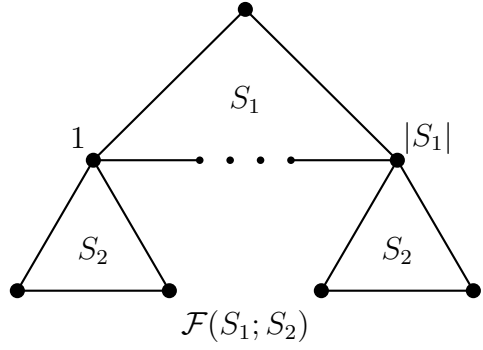

Starting with $T^{[1]}$ being the single leaf, we define recursively the family $T^{[n+1]}=\mathcal{F}(D;T^{[n]})$ of strictly $d$-ary trees. It is clear that  $T^{[2]}=D$ and  $ |T^{[n]}|=|D|^{-1+n}$. In this fashion, there are $|D| \cdot c(D,T^{[n]})$ copies of $D$ in the $|D|$ copies of $T^{[n]}$ that are attached to the leaves of $D$ to obtain $T^{[n+1]}$. Also, if one takes one arbitrary leaf from each of these $|D|$ copies, one obtains a total of another $|T^{[n]}|^{|D|}$ copies of $D$. Thus, we get
\begin{align*}
c(D,T^{[1+n]}) \geq |D| \cdot c(D,T^{[n]}) + |T^{[n]}|^{|D|}\,,
\end{align*}
which gives
\begin{align*}
c(D,T^{[1+n]}) \geq \sum_{j=0}^{n-1} |D|^j \cdot |T^{[-j+n]}|^{|D|}
\end{align*}
by iteration 
 as $c(D, T^{[1]}) =0$. Since $ |T^{[-j+n]}|=|D|^{-1-j+n}$, we establish the inequality
\begin{align*}
c(D,T^{[n]}) \geq |D|^{(-2+n)|D|} \cdot \sum_{j=0}^{n-2} |D|^{j(1-|D|)}\,,
\end{align*}
which becomes
\begin{align*}
 \gamma\big(D,T^{[n]}\big) \geq  \frac{|D|^{(-1+n)|D|} - |D|^{-1+n}}{(|D|^{|D|} -|D|) \binom{|D|^{-1+n}}{|D|}}\,.
\end{align*}
Letting $n \to \infty$ produces
\begin{align*}
\limsup_{n\to \infty} \gamma\big(D,T^{[n]}\big) \geq \frac{(-1+|D|)!}{-1+|D|^{-1+|D|}}
\end{align*}
and, in particular, the assertion of the proposition. 
\end{proof}

The bound obtained in Proposition~\ref{howsamllJS} seems to be weak in general. Possibly, $C_d$ is even the only $d$-ary tree that attains this bound (see Theorem~\ref{Cd}).

\medskip
The tree $\mathcal{F}(S_1;S_2)$ (as constructed in the proof of Proposition~\ref{howsamllJS}) offers another special result:

\begin{theorem}
Let $d\geq 2$ be an arbitrary but fixed positive integer. For two $d$-ary trees $S_1,S_2$, let $\mathcal{F}(S_1;S_2)$ be as defined in the proof of Proposition~\ref{howsamllJS}. Then we have
\begin{align*}
i_d\big(\mathcal{F}(S_1;S_2) \big) \geq \frac{(|S_1|\cdot |S_2|)!}{(|S_2|!)^{|S_1|}\cdot |S_1|^{|S_1|\cdot |S_2|}} \cdot \big(i_d(S_2)\big)^{|S_1|}\,.
\end{align*}
\end{theorem}

\begin{proof}
Fix $d\geq 2$. For a $d$-ary tree $T$, it is easy to see that a copy of $\mathcal{F}(S_1;S_2)$ in $\mathcal{F}(S_1;T)$ can be obtained by taking one copy of $S_2$ in $T$ for each of the $|S_1|$ copies of $T$ planted in $\mathcal{F}(S_1;T)$. Thus, the inequality
\begin{align*}
c\big(\mathcal{F}(S_1;S_2),\mathcal{F}(S_1;T) \big) \geq c\big(S_2,T\big)^{|S_1|}
\end{align*}
is valid so that we obtain
\begin{align*}
\max_{\substack{|T^{\prime}|=|S_1| \cdot |T| \\ T^{\prime}~\text{$d$-ary tree}}} \frac{c\big(\mathcal{F}(S_1;S_2),\mathcal{F}(S_1;T^{\prime}) \big)}{\binom{|S_1| \cdot |T|}{|S_1| \cdot |S_2|}} \geq  \frac{\binom{|T|}{|S_2|}^{|S_1|}}{\binom{|S_1| \cdot |T|}{|S_1| \cdot |S_2|}} \cdot \gamma\big(S_2,T\big)^{|S_1|}\,.
\end{align*}
Now we choose $T$ in such a way that
\begin{align*}
|T|=n,~c(S_2,T)=\max_{\substack{|T^{\prime}|=n\\T^{\prime}~\text{$d$-ary tree}}} c\big(S_2,T^{\prime}\big)\,.
\end{align*}
Letting $n \to \infty$ yields 
\begin{align*}
\limsup_{n \to \infty}\max_{\substack{|T^{\prime}|=|S_1| \cdot n \\ T^{\prime}~\text{$d$-ary tree}}}
\gamma\big(\mathcal{F}(S_1;S_2),\mathcal{F}(S_1;T^{\prime}) \big) \geq \frac{(|S_1|\cdot |S_2|)!}{(|S_2|!)^{|S_1|}\cdot |S_1|^{|S_1|\cdot |S_2|}} \cdot \big(i_d(S_2)\big)^{|S_1|}\,.
\end{align*}
Consequently, we obtain the desired inequality.
\end{proof}

In some special cases, there is an upper bound that asymptotically matches the lower bound on $i_d\big(\mathcal{F}(S_1;S_2) \big)$ as $|\mathcal{F}(S_1;S_2)|$ gets large---see \cite{ClassDossouOloryWagner}.

\section{The Trees with the Maximal Inducibility} \label{maxim}

The first natural question to pose regarding a graph invariant concerns its extreme values and the extremal graphs. In our case, we have already proved that every tree has positive inducibility (Proposition~\ref{howsamllJS}) and that $i_d(F^2_k)=1$ for every $k$ and every $d$ (Theorem~\ref{inducibility of the binary caterpillar is 1 in d ary trees}). But besides binary caterpillars, are there other $d$-ary trees with inducibility $1$? The answer turns out to be negative. This extends the result for binary trees in~\cite{czabarka2016inducibility}.

\begin{theorem}\label{In d ary trees, binary caterpillars always form a positive proportion of all trees induced by $k$ leaves in the limit}
Let $d\geq 2$ be an arbitrary but fixed positive integer. Among $d$-ary trees, only binary caterpillars have inducibility $1$.
\end{theorem}

\begin{proof}
Fix $d\geq 2$ and let us prove that for any fixed positive integers $k>1$ and $n>d^{k-2}$, every $d$-ary tree with $n$ leaves contains a copy of the $k$-leaf binary caterpillar $F^2_k$.

The key observation is that if a $d$-ary tree has height at least $k-1$, then it must have $F^2_k$ as a leaf-induced subtree. So fix $k>1$, $n>d^{k-2}$, and consider a $d$-ary tree $T$ with $n$ leaves. It is easy to see by induction on $h$ that for every fixed integer $h\geq 0$, a $d$-ary tree with height at most $h$ can never have more than $d^h$ leaves. Since $T$ has more than $d^{k-2}$ leaves, its height must be at least $k-1$, so it contains $F^2_k$. In conclusion, for $k>1$, every $n$-leaf $d$-ary tree must contain a copy of the binary caterpillar $F^2_k$ as soon as $n>d^{k-2}$.

\medskip
Now, consider a $d$-ary tree $D$ that has at least three leaves. If $D$ is different from $F^2_{|D|}$, then we obtain $c(D,T)<\binom{|T|}{|D|}$ for every $d$-ary tree $T$ with $n$ leaves provided that $n>d^{k-2}$. Therefore, fixing $n>d^{k-2}$, we get
\begin{align*}
\max_{\substack{|T|=n\\ T~\text{$d$-ary tree}}} \gamma(D,T) <1\,.
\end{align*}
On the other hand, from the second part of Theorem~\ref{maxdensityId}, we have
\begin{align*}
i_d(D)\leq \max_{\substack{|T|=n\\ T~\text{$d$-ary tree}}} \gamma(D,T)\,,
\end{align*}
which now completes the proof of the theorem.
\end{proof}

As one would expect, the limit
\begin{align*}
\liminf_{\substack{|T| \to \infty \\ T~\text{$d$-ary tree}}} \gamma(F^2_k,T)
\end{align*}
is positive for every $d$ and every $k$. An exact formula is available in \cite{ClassDossouOloryWagner}.

\section{Concluding Comments}

It seems appropriate to close with further problems on the inducibility of $d$-ary trees.

Since we already have a general lower bound on $i_d(D)$, it would be nice to understand the following problem:
\begin{problem}
Given positive integers $d\geq 2$ and $k\geq 5$, find
\begin{align*}
\min_{\substack{|D|=k\\D~\text{$d$-ary tree}}} i_d(D)\,,
\end{align*}
and furthermore, characterize the $d$-ary trees that attain this minimum.
\end{problem}

We conjecture that except for binary trees, the inducibility in $d$-ary trees always depends on $d$ (in contrast to binary trees, where it does not; see Theorem~\ref{Upper bound for all binary trees} and the discussion thereafter).

\begin{conjecture}
Let $d\geq 2$ be an arbitrary but fixed positive integer. Among all $d$-ary trees $D$, only binary trees satisfy
\begin{align*}
i_d(D)=i_{d+1}(D)=i_{d+2}(D)=\cdots.
\end{align*}
\end{conjecture}

\end{document}